\documentclass{amsart}

\usepackage{graphicx}

\newtheorem{theorem}{Theorem}[section]
\newtheorem{lemma}[theorem]{Lemma}
\newtheorem{proposition}[theorem]{Proposition}
\newtheorem{corollary}[theorem]{Corollary} 

\theoremstyle{definition}
\newtheorem{definition}[theorem]{Definition}
\newtheorem{example}[theorem]{Example}

\theoremstyle{remark}
\newtheorem{remark}[theorem]{Remark}
\newtheorem{problem}[theorem]{Problem}

\numberwithin{equation}{section}

\begin{document}

\title{Combinatorial constructions of three-dimensional Small Covers}

\author{Yasuzo NISHIMURA}
\address{Department of Mathematics,
Faculty of Education and Regional Studies, University of Fukui,
3-9-1 Bunkyo, Fukui 910-8507, Japan}
\email{y-nishi@u-fukui.ac.jp}


\date{March 24, 2011.}


\begin{abstract}
In this paper we study two operations on $3$-dimensional
small covers called a connected sum and a surgery.
These operations correspond to combinatorial operations on 
$(\mathbb{Z}_2)^3$-colored simple convex polytopes.
Then we show that each $3$-dimensional small cover can be
constructed from $T^3$, $\mathbb{R}P^3$ and $S^1 \times \mathbb{R}P^2$ with
two different $(\mathbb{Z}_2)^3$-actions by using these operations.
This result is a generalization or an improvement of results in \cite{I}, 
\cite{K}, \cite{LY} and \cite{N}.
\end{abstract}

\maketitle


\section{Introduction}
A small cover was introduced by Davis and Januszkiewicz \cite{DJ}
as an $n$-dimensional closed manifold $M^n$ with a locally standard
$(\mathbb{Z}_2)^n$-action such that its orbit space is a simple convex 
polytope $P$ where $\mathbb{Z}_2$ is the quotient additive group
$\mathbb{Z}/2 \mathbb{Z}$.
They showed that there exists a one-to-one correspondence
between small covers and $(\mathbb{Z}_2)^n$-colored polytopes 
(cf. \cite[Proposition 1.8]{DJ}).
Here a pair $(P, \lambda)$ is called a 
{\it $(\mathbb{Z}_2)^n$-colored polytope} when $P$ is an $n$-dimensional 
simple convex polytope with the set of facets $\mathcal{F}$
and a function $\lambda : \mathcal{F} \to (\mathbb{Z}_2)^n$ 
satisfying the following condition:

\vspace{10pt}
$(\star)$ if $F_1 \cap \cdots \cap F_n \neq \emptyset$ then
$\{ \lambda(F_1), \cdots, \lambda(F_n) \}$ is linearly independent.
\vspace{10pt}

We say that two $(\mathbb{Z}_2)^n$-colored polytopes $(P_i, \lambda_i)~(i=1,2)$
are {\it equivalent} when there exists a combinatorial equivalence 
of polytopes $\phi:P_1 \to P_2$ such that $\lambda_2 \phi=\theta \lambda_1$ 
for some $\theta \in {\rm Aut}(\mathbb{Z}_2)^n$.
The $n$-dimensional torus $T^n$ and the real projective space
$\mathbb{R}P^n$ with standard $(\mathbb{Z}_2)^n$-actions are
examples of small covers over the $n$-cube $I^n$ and the $n$-simplex 
$\Delta^n$ respectively.

In this paper we are interested in constructions of $3$-dimensional 
small covers $M^3$ from basic small covers by using some operations.
In \cite{I} Izmestiev studied a class of $3$-dimensional small covers
which are called {\it linear models} and are correspondent 
to $3$-colored polytopes.
He introduced two operations on linear models called a
{\it connected sum} $\sharp$ and a {\it surgery} $\natural$ 
and proved the following theorem (cf. \cite[Theorem 3]{I}).

\begin{theorem}[Izmestiev]
Each linear model $M^3$ can be constructed from 
$T^3$ by using three operations
$\sharp$, $\natural$ and $\natural^{-1}$ where $\natural^{-1}$ 
is the inverse of $\natural$.
\end{theorem}

In \cite{N} we generalized Theorem 1.1 to orientable small covers $M^3$
which are correspondent to $4$-colored polytopes.
We introduced a new operation called the Dehn surgery $\natural^D$,
and showed that each {\it orientable} small cover $M^3$ can be
constructed from $T^3$ and $\mathbb{R}P^3$ by using four
operations $\sharp$, $\natural$, $\natural^{-1}$ and $\natural^D$ 
(cf. \cite[Theorem 1.10]{N}).
Later L\"{u} and Yu \cite{LY} considered a construction of general small covers
$M^3$.
They introduced new operations $\sharp^e$, $\sharp^{eve}$, $\sharp^\Delta$ and
$\sharp^\copyright_i$ $(i \ge 3)$ and showed the following theorem 
(cf. \cite[Theorem 1.2]{LY}).

\begin{theorem}[L\"{u} and Yu]
Each small cover $M^3$ can be constructed from $\mathbb{R}P^3$ and 
$S^1 \times \mathbb{R}P^2$ with a certain $(\mathbb{Z}_2)^3$-action
by using seven operations $\sharp$, $\natural^{-1}$, $\sharp^e$, 
$\sharp^{eve}$, $\sharp^\Delta$, $\sharp^\copyright_4$ and
$\sharp^\copyright_5$.
\end{theorem}

Operations appeared in Theorem 1.2 are all ``non-decreasing'' 
i.e. they do not decrease the number of faces of an orbit polytope, 
and therefore the use of the surgery $\natural$ 
is prohibited unlike Theorem 1.1.
In \cite{K} Kuroki pointed out that the operations
$\natural^D$, $\sharp^e$ and $\sharp^{eve}$ can be obtained 
as compositions of $\sharp$ and $\natural$
such as $\natural^D = \natural \circ \sharp \mathbb{R}P^3$, 
$\sharp^e= \natural \circ \sharp$ and $\sharp^{eve}=\natural^2 \circ \sharp$,
respectively (cf. \cite[Theorem 4.1]{K}).
Therefore our result in \cite{N} can be improved as follows:
Each orientable small cover $M^3$ can be constructed from $\mathbb{R}P^3$ and
$T^3$ by using three operations $\sharp$, $\natural$ and $\natural^{-1}$.
(cf. \cite[Corollary 4.4]{K}).
Moreover L\"{u}-Yu's result can be rewritten by using $\natural$ instead of 
$\sharp^e$ and $\sharp^{eve}$ as follows (cf. \cite[Corollary 4.8]{K}):
Each small cover $M^3$ can be constructed from $\mathbb{R}P^3$ and
$S^1 \times \mathbb{R}P^2$ with a certain $(\mathbb{Z}_2)^3$-action 
by using six operations $\sharp$, $\natural$, $\natural^{-1}$, 
$\sharp^\Delta$, $\sharp^\copyright_4$ and $\sharp^\copyright_5$.
Then a problem arises (cf. \cite[Problem 5.2]{K}).

\begin{problem}
What are basic small covers from which we can construct all $3$-dimensional 
small covers using the three operations $\sharp$, $\natural$ and 
$\natural^{-1}$ ? 
\end{problem}

We give a solution to this problem.
The following is our main result.

\begin{theorem}\label{ThN2}
Each small cover $M^3$ can be constructed from $T^3$, $\mathbb{R}P^3$ and
$S^1 \times \mathbb{R}P^2$ with two different $(\mathbb{Z}_2)^3$-actions
by using two operations $\sharp$ and $\natural$.
\end{theorem}

In the above theorem we do not use the inverse surgery $\natural^{-1}$.
As a corollary we obtain improvements of Theorem 1.1 and
our previous result in \cite{N}.

\begin{corollary}\label{CorN2}
(1) Each linear model $M^3$ can be constructed from $T^3$ 
by using two operations $\sharp$ and $\natural$.\\
(2) Each orientable small cover $M^3$ can be constructed from $T^3$ and
$\mathbb{R}P^3$ by using two operations $\sharp$ and $\natural$.
\end{corollary}

These results are equivariant analogues of a well-known result (cf. \cite{Ki}):
``Each closed $3$-manifold can be constructed from the $3$-sphere $S^3$
by using the Dehn surgeries''.

This paper is organized as follows.
In section 2 we recall the definition and the basic facts about small covers
briefly, and we introduce some basic $3$-dimensional small covers.
In section 3 we establish several operations on $(\mathbb{Z}_2)^3$-colored
polytopes.
In section 4 we discuss the constructions of $(\mathbb{Z}_2)^3$-colored
polytopes, and prove Theorem \ref{ThN2}.
In section 5 we follow the standpoint of L\"{u} and Yu, and discuss 
a non-decreasing construction of small covers
by using the inverse surgery $\natural^{-1}$ instead of 
the decreasing surgery $\natural$.
We shall point out that there is a gap in the proof of Theorem 1.2 in \cite{LY}
(Remark \ref{gap}) and improve their result as follows.

\begin{theorem}\label{ThLast}
(1) Each linear model $M^3$ can be constructed from $T^3$
by using three operations $\sharp$, $\sharp^e$ and $\natural^{-1}$.\\
(2) Each orientable small cover $M^3$ can be constructed from 
$T^3$ and $\mathbb{R}P^3$ 
by using three operations $\sharp$, $\sharp^e$ and $\natural^{-1}$.\\
(3) Each small cover $M^3$ can be constructed from 
$\mathbb{R}P^3$ and $S^1 \times \mathbb{R}P^2$ with two different
$(\mathbb{Z}_2)^3$-actions 
by using four operations $\sharp$, $\sharp^e$, $\natural^{-1}$
and $\sharp^\copyright_4$.
\end{theorem}

In section 6 we shall make a remark on a {\it $2$-torus manifold} 
which is an object of a little wider class than small covers.
If the object is expanded to this class, the argument becomes 
easier.
We prove the following theorem.

\begin{theorem}\label{NN2}
(1) Each linear model of a locally standard $2$-torus manifold over $D^3$ 
can be constructed from $S^3$ by using inverse surgery $\natural^{-1}$.\\
(2) Each orientable locally standard $2$-torus manifold over $D^3$ 
can be constructed from $S^3$ by using two surgeries $\natural^{-1}$,
$\natural^D$ and the blow up $\sharp \mathbb{R}P^3$.\\
(3) Each locally standard $2$-torus manifold over $D^3$ 
can be constructed from $S^3$ by using the inverse surgery $\natural^{-1}$
and connecting $\mathbb{R}P^3$, $S^1 \times_{\mathbb{Z}_2} S^2$, 
$S^1 \times \mathbb{R}P^2$ with certain $(\mathbb{Z}_2)^3$-actions 
by operations $\sharp$ and $\sharp^e$.
\end{theorem}

\section{Basics of small covers}
In this section we recall the definitions and basic facts on small covers 
(see \cite{DJ} for detail).
Let $P$ be an $n$-dimensional simple convex polytope with facets 
(i.e., codimension-one faces) $\mathcal{F}= \{ F_1, \cdots, F_m \}$.
A {\it small cover} $M$ over $P$ is an $n$-dimensional closed manifold
with a locally standard $(\mathbb{Z}_2)^n$-action such that 
its orbit space is $P$.
For a facet $F$ of $P$,
we define $\lambda (F)$ to be the generator of the isotropy subgroup at
$x \in \pi^{-1}({\rm int} F)$ where $\pi :M \to P$ is the orbit projection.
Then a function $\lambda : \mathcal{F} \to (\mathbb{Z}_2)^n$ is called a
{\it characteristic function} of $M$ which satisfies the following condition.
\vspace{10pt}

$(\star)$ if $F_1 \cap \cdots \cap F_n \neq \emptyset$ 
then $\{ \lambda(F_1), \cdots, \lambda(F_n) \}$ is linearly independent.
\vspace{10pt}

Therefore $\lambda$ is a kind of face-coloring of $P$.
Then we call a function satisfying $(\star)$ 
a {\it $(\mathbb{Z}_2)^n$-coloring} of $P$.
We say that two $(\mathbb{Z}_2)^n$-colored polytopes $(P_i, \lambda_i)~(i=1,2)$
are {\it equivalent} when there exists a combinatorial equivalence 
of polytopes $\phi:P_1 \to P_2$ such that $\lambda_2 \phi=\theta \lambda_1$ 
for some $\theta \in {\rm Aut}(\mathbb{Z}_2)^n$.
Conversely, given a simple convex polytope $P$ and a $(\mathbb{Z}_2)^n$-coloring
$\lambda: \mathcal{F} \to (\mathbb{Z}_2)^n$ satisfying $(\star)$, 
we can construct a small cover $M$ 
such that its characteristic function is the given $\lambda$ as follows:
$$M(P, \lambda) := P \times (\mathbb{Z}_2)^n / \sim,$$
where $(x,t) \sim (y, s)$ is defined as $x=y \in P$ and
$s-t$ is contained in the subgroup generated by 
$\lambda(F_1), \cdots, \lambda(F_k)$
such that $x \in {\rm int}(F_1 \cap \cdots \cap F_k)$.
We say that two small covers $M_i$ over $P_i$ $(i=1,2)$ are 
{\it ${\rm GL}(n, \mathbb{Z}_2)$-equivalent} on a combinatorial equivalence 
of polytopes $\phi:P_1 \to P_2$
when there exists a $\theta$-equivariant homeomorphism
$f: M_1 \to M_2$ such that $\pi_2 \circ f = \phi \circ \pi_1$
i.e. $f(g \cdot x)= \theta (g) \cdot f(x) ~(g \in (\mathbb{Z}_2)^n, ~x \in M_1)$
for some $\theta \in {\rm Aut}(\mathbb{Z}_2)^n$.
Moreover we say that two small covers are {\it equivalent} when 
they are {\it ${\rm GL}(n, \mathbb{Z}_2)$-equivalent} on some equivalence 
$\phi:P_1 \to P_2$.
In \cite{LM} this equivalence and a ${\rm GL}(n,\mathbb{Z}_2)$-equivalence
on the identity are called a {\it weakly equivariantly homeomorphism} and a
{\it D-J equivalence}, respectively.
Davis and Januszkiewicz proved that a small cover $M$ over $P$
with a characteristic function $\lambda$ is 
${\rm GL}(n, \mathbb{Z}_2)$-equivalent 
on the identity to $M(P, \lambda)$ when we fix a polytope $P$ 
(cf. \cite[Proposition 1.8]{DJ}).
Therefore we can identify the equivalence class 
of a small cover $M(P, \lambda)$ with the equivalence class of a 
$(\mathbb{Z}_2)^n$-colored polytope $(P, \lambda)$.

\begin{example}\label{ex1}
The real projective space $\mathbb{R}P^n$ and the $n$-dimensional torus $T^n$ 
with standard $(\mathbb{Z}_2)^n$-actions are examples of small covers
over the $n$-simplex $\Delta^n$ and the $n$-cube $I^n$ respectively.
Figure \ref{RP3T3.bmp} shows their characteristic functions on the
polytopes in the case $n=3$, where $\{ \alpha, \beta, \gamma \}$
is a basis of $(\mathbb{Z}_2)^3$.
We notice that a $(\mathbb{Z}_2)^n$-coloring on $\Delta^n$ is unique
up to equivalence.
Therefore we denote the colored simplex by $\Delta^n$ by omitting coloring.

\begin{figure}[htbp]
\begin{center}
\includegraphics[bb=0 0 517 209,width=5cm]{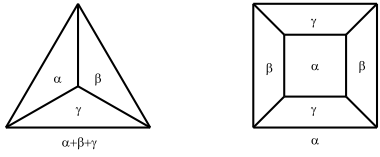}
\end{center}
\caption{Characteristic functions of $\mathbb{R}P^3$ and $T^3$.}
\label{RP3T3.bmp}
\end{figure}

\end{example}

A small cover over $P$ with $n$-coloring (i.e. $\lambda(\mathcal{F})$
is a basis of $(\mathbb{Z}_2)^n$) is called a {\it linear model}.
An example of a linear model is the torus $T^n$ shown in Example \ref{ex1}.
In this case the $n$-coloring of $P$ (i.e. the  linear model) 
is unique up to equivalence.
In case $n=3$, it is well-known that a simple convex polytope 
is $3$-colorable if and only if each face contains an even number of edges.

In \cite[Theorem 1.7]{NN}, we gave a criterion of when a small cover is 
orientable.
We recall the criterion in the case $n=3$.

\begin{theorem}
A $3$-dimensional small cover $M(P, \lambda)$ is orientable if and only if
$\lambda (\mathcal{F})$ is contained in $\{ \alpha, \beta, \gamma, \alpha+\beta+\gamma \}$
for a suitable basis $\{ \alpha, \beta, \gamma \}$ of $(\mathbb{Z}_2)^3$.
\end{theorem}

From the above theorem the small covers $\mathbb{R}P^3$ and $T^3$ 
given in Figure 1 are both orientable.
We call a $(\mathbb{Z}_2)^3$-coloring satisfying the orientability 
condition in the above theorem {\it an orientable coloring} of $P$.
Since each triple of $\{ \alpha, \beta, \gamma, \alpha+\beta+\gamma \}$ is
linearly independent, the orientable coloring is just an ordinary $4$-coloring.

\begin{example}\label{Ex:prism}
We consider small covers on the $3$-sided prism $P^3(3)=I \times \Delta^2$.
There exist three types of $(\mathbb{Z}_2)^3$-coloring 
on $P^3(3)$ shown in Figure \ref{prism.bmp} up to equivalence.
The first example $M(P^3(3), \lambda_1)$ is non-equivariantly homeomorphic
to $S^1 \times \mathbb{R}P^2$.
The second example $M(P^3(3), \lambda_2)$
is not equivariantly homeomorphic to $M(P^3(3),\lambda_1)$ but non-equivariantly homeomorphic to 
$S^1 \times \mathbb{R}P^2$ (cf. \cite[Lemmas 4.2 and 4.3]{LY}).
The last example $M(P^3(3), \lambda_3)$ is orientable and
homeomorphic to $\mathbb{R}P^3 \sharp \mathbb{R}P^3$
where $\sharp$ is the connected sum (see the following section).

\begin{figure}[htbp]
\begin{center}
\includegraphics[bb=0 0 452 195,width=5cm]{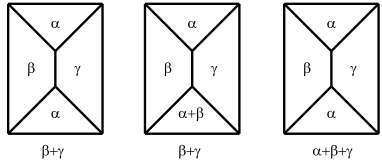}
\end{center}
\caption{Basic three types of $(\mathbb{Z}_2)^3$-coloring on
$3$-sided prism $P^3(3)=I \times \Delta^2$; $\lambda_1, \lambda_2$
and $\lambda_3$ respectively.}
\label{prism.bmp}
\end{figure}

\end{example}

\begin{example}\label{Ex:cube}
It is easily verified that 
there exist four types of $(\mathbb{Z}_2)^3$-coloring on the $3$-cube 
$I^3=P^3(4)$.
One of them is the $3$-colored cube which is already seen in 
Figure \ref{RP3T3.bmp}, and is denoted by $(I^3,\lambda_0)$.
The other three types are shown in Figure \ref{cube.bmp}.
The associated small covers are homeomorphic to 
$S^1 \times K$, a twisted $K$-bundle over $S^1$ and
a twisted $T^2$-bundle over $S^1$ according to $\lambda_1$,
$\lambda_2$ and $\lambda_3$ respectively,
where $K=\mathbb{R}P^2 \sharp \mathbb{R}P^2$ is the Klein's bottle
(more precisely see \cite[Lemmas 5.3 and 5.4]{LY}).

\begin{figure}[htbp]
\begin{center}
\includegraphics[bb=0 0 540 178,width=6cm]{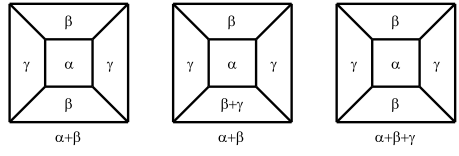}
\end{center}
\caption{Three types of $(\mathbb{Z}_2)^3$-coloring on
$3$-cube $I^3$; $\lambda_1$, $\lambda_2$ 
and $\lambda_3$ respectively
(except the $3$-colored cube in Figure \ref{RP3T3.bmp}).}
\label{cube.bmp}
\end{figure}
\end{example}

\begin{remark}
In \cite{LY} the ${\rm GL}(3, \mathbb{Z}_2)$-equivalence 
on the identity (D-J equivalence) is adopted as an equivalence relation of 
$(\mathbb{Z}_2)^3$-colored polytopes,
i.e. $(P, \lambda) \sim (P, \theta \lambda)$ for $\theta \in
{\rm Aut}(\mathbb{Z}_2)^3$.
Therefore it is written that there exist five (resp. seven) types of 
$(\mathbb{Z}_2)^3$-coloring on $P^3(3)$ (resp. $I^3$) in \cite{LY}.
Discussing D-J equivalence classes only when the orbit polytope $P$ 
is fixed has the meaning.
However, the orbit polytopes will be not fixed in the following sections.
Then we adopt our equivalence (the weakly equivariantly homeomorphism) 
instead of the D-J equivalence.
In this paper we shall rewrite results in \cite{LY}
to our standpoint by our equivalence.
The difference between the D-J equivalence
and our equivalence is not essential in the discussion of the following
sections.
\end{remark}


\section{Operations on small covers}
Henceforth we assume that $n=3$ and $(P, \lambda)$ is a pair of
a $3$-dimensional simple convex polytope $P$ with a $(\mathbb{Z}_2)^3$-coloring
$\lambda$, and $\{ \alpha, \beta, \gamma \}$ is a basis of $(\mathbb{Z}_2)^3$.
We call a $3$-dimensional simple convex polytope a {\it $3$-polytope}
for simplicity.
From the Steinitz's theorem (see  \cite{G} etc.)  combinatorially equivalence 
classes of  $3$-polytopes bijectively correspond to $3$-connected 
$3$-valent simple planner graphs i.e. $1$-skeleton of $P$.
Here a graph $\Gamma$ is called {\it $k$-connected}, {\it $l$-valent}
and {\it simple} if $\Gamma$ is connected after cutting any $(k-1)$ edges,
the degree of each vertex is $l$, and there is no loop and no multi-edge,
respectively.
In this section we recall some operations on $(\mathbb{Z}_2)^3$-colored 
polytopes (or  small covers),
which were introduced in \cite{I}, \cite{LY} and \cite{N}.

\begin{definition}[the connected sum $\sharp$]
The operation $\sharp$ in Figure \ref{connect.bmp} (from left to right) is called the 
{\it connected sum (at vertices)} 
 and its inverse (from right to left) is denoted by $\sharp^{-1}$.
These operations also can be defined for non-colored polytopes.
Remark that $P_1 \sharp P_2$ is also a $3$-polytope 
for any $3$-polytopes $P_i~(i=1,2)$ from the Steinitz's theorem.
The operation $\sharp$ corresponds to the connected sum
$M(P_1, \lambda_1) \sharp M(P_2, \lambda_2)$ around fixed points of them
(cf. \cite[1.11]{DJ} or \cite[Definition 3]{I}).
We say that $(P, \lambda)$ is {\it decomposable} 
(as a $(\mathbb{Z}_2)^3$-colored polytope) when there exist two 
$(\mathbb{Z}_2)^3$-colored polytopes $(P_i, \lambda_i) ~(i=1,2)$ 
such that $(P, \lambda) = (P_1, \lambda_1) \sharp (P_2, \lambda_2)$.
Similarly we say that $P$ is {\it decomposable as a non-colored polytope}
when $P=P_1 \sharp P_2$ as non-colored polytopes for some $P_i ~(i=1,2)$.
\end{definition}

\begin{figure}[htbp]
\begin{center}
\includegraphics[bb=0 0 537 163,width=6cm]{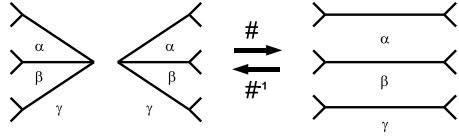}
\end{center}
\caption{The connected sum $\sharp$ and its inverse $\sharp^{-1}$.}
\label{connect.bmp}
\end{figure}

Specifically the connected sum with $\Delta^3$ on polytopes, denoted by
$\sharp \Delta^3$ (and often called a {\it cutting vertex} or 
{\it bistellar $0$-move}), corresponds to the operation
called a {\it blow up} on small covers (Figure \ref{blowup.bmp}).
Its inverse $\sharp^{-1}\Delta^3$ (often called a {\it bistellar $2$-move}) 
is called a {\it blow down}.

\begin{figure}[htbp]
\begin{center}
\includegraphics[bb=0 0 418 163,width=5cm]{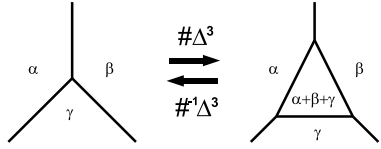}
\end{center}
\caption{The blow up $\sharp \Delta^3$ and the blow down 
$\sharp^{-1}\Delta^3$.}
\label{blowup.bmp}
\end{figure}

\begin{definition}[the surgery $\natural$]
The operation $\natural$ in Figure \ref{surgery.bmp} (from left to right)
is called the {\it surgery} along an edge $e$ 
and its inverse $\natural^{-1}$ (from right to left) is called the {\it inverse surgery}
along a pair of edges $e_1$ and $e_3$.
The operations $\natural$ and $\natural^{-1}$ both correspond to the 
ordinaly surgeries on small covers (cf. \cite{I}).
In the previous papers \cite{I}, \cite{K}, \cite{LY} and \cite{N},
surgeries $\natural$ and $\natural^{-1}$ were not distinguished 
and they both were denoted by the same symbol $\natural$.

\begin{figure}[htbp]
\begin{center}
\includegraphics[bb=0 0 474 167,width=6cm]{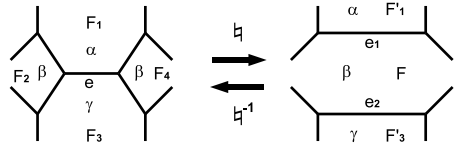}
\end{center}
\caption{The suegery $\natural$ and its inverse $\natural^{-1}$.}
\label{surgery.bmp}
\end{figure}

\end{definition}

We do not allow the surgeries $\natural$ and $\natural^{-1}$
when the $3$-connectedness of the $1$-skeleton
of $P$ is destroyed after doing it, i.e. the following cases respectively:

\begin{description}
\item[in case $\natural$] if and only if $F_2$ and $F_4$ are adjacent to a same face except $F_1$ and $F_3$
(involve the case when $F_1$ or $F_3$ is a quadrilateral),
\item[in case $\natural^{-1}$] if and only if $F'_1$ is adjacent to $F'_3$.
\end{description}

\begin{definition}[the connected sum along edges $\sharp^e$]
The operation $\sharp^e$ in Figure \ref{connedge.bmp} 
(from left to right) is called the 
{\it connected sum along edges} and its inverse is denoted by $(\sharp^e)^{-1}$.
We notice that the operation $\sharp^e$ is obtained as the composition $\sharp^e=\natural 
\circ \sharp$ as shown in the same figure (cf. \cite[Theorem 4.1(2)]{K}).
The operation $\sharp^e$ corresponds to the connected sum along the circle
$\pi^{-1}(e)$ on a small cover $M$ where $\pi:M \to P$ is the projection
(cf. \cite{LY}).

\begin{figure}[htbp]
\begin{center}
\includegraphics[bb=0 0 503 364,width=8cm]{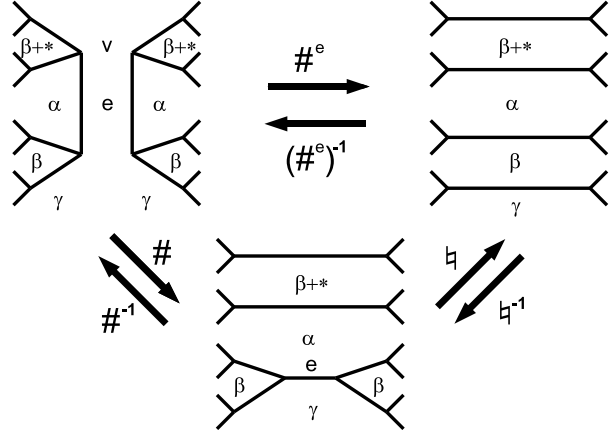}
\end{center}
\caption{The connected sum along the edges $\sharp^e$ and its inverse
$(\sharp^e)^{-1}$. The figure also shows that 
$\sharp^e= \natural \circ \sharp$.}
\label{connedge.bmp}
\end{figure}

Specifically the operations $\sharp^e P^3(3)$
(along a vertical edge in Figure \ref{prism.bmp}) and 
$\sharp^e \Delta^3$ are often called the {\it cutting edge} and 
the {\it bistellar $1$-move}, respectively (Figure \ref{cutedge.bmp}).
The former (left diagram) corresponds to a blow up 
along the circle $\pi^{-1}(e)$ on a small cover.
In this diagram we can choose not only $\beta + \gamma$ but also 
$\alpha + \beta + \gamma$ as a color of the center square when $\ast =0$.
The latter operation $\sharp^e \Delta^3  = \natural \circ \sharp \Delta^3$ 
corresponds to the Dehn surgery of type $\frac{2}{1}$ on a small cover 
(cf. \cite{N} or \cite[3.5]{K}).
This operation is denoted by $\natural^D$ and is called the {\it Dehn surgery}.
This operation can be done along an edge $e$ which satisfies 
the following condition:
$$\sum_e \lambda(F):= 
\sum_{\{ F \in \mathcal{F}|e \cap F \neq \emptyset \}} \lambda(F)=0.$$
We call such an edge {\it $0$-sum edge} (or {\it $4$-colored edge} 
in orientable case).
We notice that the Dehn surgery $\natural^D$ does not 
change the number of faces, and is invertible because
$(\natural^D)^{-1}=\natural^D$.
\end{definition}

\begin{figure}[htbp]
\begin{center}
\includegraphics[bb=0 0 564 152,width=9.5cm]{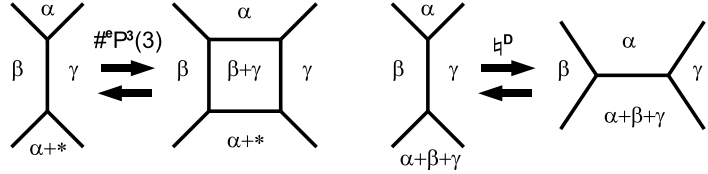}
\end{center}
\caption{The cutting edge $\sharp^e P^3(3)$ and the Dehn sugery $\natural^D
= \sharp^e \Delta^3$.}
\label{cutedge.bmp}
\end{figure}

From the Steinitz's theorem, a $3$-polytope $P$ is decomposable as a 
non-colored polytope if and only if there exist three edges such that
they are not adjacent to each other and the $1$-skeleton of $P$ 
becomes disconnected after cutting them.
Obviously if an orientable ($4$-)colored polytope $P$
is decomposable as a non-colored polytope then
$(P, \lambda)$ is also decomposable (as a $(\mathbb{Z}_2)^3$-colored 
polytope).
However we need a little attention for non-orientable colored polytopes.
We say that $(P, \lambda)$ is {\it quasi-decomposable} when
there exist two $(\mathbb{Z}_2)^3$-colored polytopes
$(P_1, \lambda_1)$ and $(P_2, \lambda_2)$
such that either 
$(P, \lambda)=(P_1, \lambda_1) \sharp (P_2, \lambda_1)$ or 
$(P, \lambda)= (P_1, \lambda_1) \sharp^e (P_2, \lambda_2)$,
except $P=P_1 \sharp^e \Delta^3 (=\natural^D P_1)$.

\begin{remark}\label{Rem:connD}
Notice that if a $1$-skeleton of $P$ becomes disconnected
after cutting three edges $\{ e', e'', e''' \}$ then these three edges 
are not adjacent to each other or meet at a vertex.
In fact if a pair $\{ e', e'' \}$ of these three edges is adjacent to each
other and the other edge $e'''$ is not adjacent to 
$e' \cap e''$ then the $1$-skeleton of $P$
becomes disconnected after cutting the edge $e'''$ and the edge which
is adjacent to $e' \cap e''$ and different from $e'$ and $e''$.
This contradicts the $3$-connectedness of the $1$-skeleton of $P$.
\end{remark}

\begin{proposition}\label{Pro:connD}
Let $(P, \lambda)$ be a $(\mathbb{Z}_2)^3$-colored polytope, 
but not $P^3(3)$.
If $P$ is decomposable as a non-colored polytope then 
$(P, \lambda)$ is quasi-decomposable.
\end{proposition}

\begin{proof}
It is sufficient to treat the case that $P$ is indecomposable
as a $(\mathbb{Z}_2)^3$-colored polytope.
Since $P$ is decomposable as a non-colored polytope,
there exist three non-adjacent edges such that $P$ becomes disconnected 
after cutting them out, and colors of the three faces adjacent to 
the these edges are not linearly independent as shown in Figure \ref{connD.bmp}.

\begin{figure}[htbp]
\begin{center}
\includegraphics[bb=0 0 542 193,width=8cm]{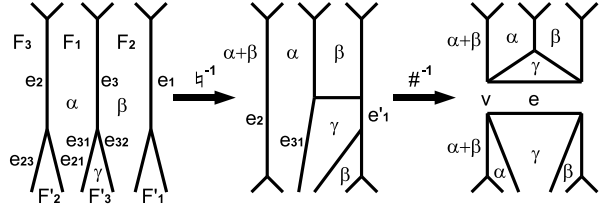}
\end{center}
\caption{The decomposition of a polytope along a $3$-cycle of 
$2$-independent faces.}
\label{connD.bmp}
\end{figure}

Since $P \neq P^3(3)$, $P$ has at least six faces so
we may assume that there are at least
two faces under the pillar ($F_i$'s) in the first diagram.
We first assume that $F_2'=F_3'$ (equivalently $e_{21}=e_{31}$ 
because if it is not so, the $1$-skeleton of $P$ becomes disconnected 
after cutting these two edges).
Then the $1$-skeleton of $P$ becomes disconnected after cutting three edges 
$e_1$, $e_{23}$ and $e_{32}$.
Since $P$ is indecomposable, these three edges actually meet at a vertex 
$F_1' \cap F_2 \cap F_3$ (see Remark \ref{Rem:connD}).
It should be $F_1'=F_2'=F_3'$ and it is a triangle.
This contradicts the assumption that there are at least two faces 
under the pillar.
Therefore the assumption $F_2'=F_3'$ is denied, and by a similar discussion 
we can reach the conclusion that $F_i' ~(i=1,2,3)$ are different 
faces each other.
We notice that if $F_3 \cap F_3' \neq \emptyset$ then it is clear that
$F_1 \cap F_1' =F_2 \cap F_2' =\emptyset$.
Therefore we can assume that $F_3 \cap F_3' = \emptyset$
by changing the role of $F_i$'s if necessary.

Now we can do the surgery $\natural^{-1}$ for edges $e_1$ and $e_{32}$,
and decompose $P$ into two $(\mathbb{Z}_2)^3$-colored polytopes 
$P_1$ and $P_2$ by cutting three non-adjacent edges 
$e_1', e_2$ and $e_{31}$ (second and third diagrams).
Then we have $\natural^{-1} P=P_1 \sharp P_2$ or 
equivalently $P= P_1 \sharp^e P_2$.
\end{proof}

Notice that the surgery $\natural$ and the Dehn surgery $\natural^D$ 
are not allowed along an edge of a quadrilateral and a triangle, respectively,
and the inverse surgery $\natural^{-1}$ is not allowed along a pair of
adjacent edges.
The following is a key lemma to relate the surgery to the connected sum. 

\begin{lemma}\label{L2}
Let $(P,\lambda)$ be a $(\mathbb{Z}_2)^3$-colored polytope.
Suppose that the $3$-connectedness of the $1$-skeleton of $P$ is 
destroyed after doing surgeries $\natural^{-1}$ or $\natural^D$,
but not the above trivial prohibited cases.
Then $(P, \lambda)$ is quasi-decomposable. 
In particular when $(P, \lambda)$ is (orientable) $4$-colored,
$(P, \lambda)$ is decomposable as a $(\mathbb{Z}_2)^3$-colored polytope.
\end{lemma}

\begin{proof}
In consequence of Proposition \ref{Pro:connD},
it is sufficient to prove that $(P, \lambda)$ is decomposable as
a non-colored polytope.\\
(1) \textbf{in case $\natural^{-1}$:} 
When the inverse surgery $\natural^{-1}$ is not allowed
in the right diagram of Figure \ref{surgery.bmp}, 
$F'_1$ is adjacent to $F'_3$.
Then cutting the three non-adjacent edges $e_1$, $e_3$ and $F'_1 \cap F'_3$ 
makes the $1$-skeleton of $P$ disconnected. 
That is $P$ is decomposable as a non-colored polytope.\\
(2) \textbf{in case $\natural^D$:} Since 
$\natural^D=(\sharp^{-1} \Delta^3) \circ \natural^{-1}$ 
and there is no obstacle for the blow down $\sharp^{-1}\Delta^3$,
the allowance of $\natural^D$ depends only on that of $\natural^{-1}$.
\end{proof}

\section{Constructions of Small Covers}

In this section we discuss constructions of $(\mathbb{Z}_2)^3$-colored 
polytopes (i.e. small covers) by using two operations 
$\sharp$ and $\natural$.
Henceforth polytopes are considered as $(\mathbb{Z}_2)^3$-colored polytopes.
In \cite{I}, Izmestiev proved the following theorem
which is a combinatorial translation of Theorem 1.1.

\begin{theorem}[Izmestiev]\label{TI}
Each $3$-colored polytope $(P^3, \lambda)$ can be constructed from 
$(I^3, \lambda_0)$ by using three operations 
$\sharp$, $\natural$ and $\natural^{-1}$.
\end{theorem}

We start from linear models and consider constructions of 
orientable small covers (i.e. $4$-colored polytopes).
Let $F$ be an $l$-gonal face of $P$.
We say that $F$ is {\it $j$-independent} ($j=2$ or $3$) when 
the rank of $\{ \lambda (F_1), \cdots, \lambda(F_l) \}$ is $j$
where $F_1, \cdots, F_l$ are faces adjacent to $F$.
In the case of orientable small covers, a $j$-independent face
is a face such that the number of colors of adjacent faces is $j$
($j=2$ or $3$).
Similarly we say that an edge is {\it $j$-colored} ($j=3$ or $4$) when
the number of the four faces adjacent to the edge is $j$.

\begin{proposition}\label{PN1}
Each $4$-colored polytope $(P^3, \lambda)$ can be constructed from $3$-colored 
polytopes and $\Delta^3$ by using two operations $\sharp$ and $\natural^D$.
\end{proposition}

\begin{proof}
By induction on the number of faces of $P$,
it is sufficient to prove the following\\
$(\ast)$ Each $4$-colored polytope $P \neq \Delta^3$ 
can be decomposed into two polytopes after doing the Dehn surgery 
$\natural^D (=(\natural^D)^{-1})$ finitely many times.

Assume that $P$ is $4$-colored and not $\Delta^3$.
Then there exists a $3$-independent face.
Let $F$ be a $3$-independent face such that the number of its edges is minimum
among $3$-independent faces of $P$, and $k$ be this number.
We prove the above $(\ast)$ by induction on $k$.
If $k=3$ (i.e., $F$ is a triangle) then we get a colored decomposition
$P= P' \sharp \Delta^3$ immediately.
We assume $k \ge 4$. Since $F$ is a $3$-independent face,
there exists a $4$-colored edge $e$ of $F$
(see Figure \ref{pfProp.bmp}).

We notice that there exist no trianglar face of $P$ because $k \ge 4$.
If the Dehn sugery $\natural^D$ is not allowed along an edge then
$P$ decomposes into two polytopes from Lemma \ref{L2}.
Therefore we may assume that the Dehn surgery $\natural^D$ is allowed 
along every $4$-colored edge of $F$.
If the $3$-independence of $F$ is preserved under
the Dehn surgery $\natural^D$ along some edge,
then we can reduce $P$ to $\natural^D P$ which has a $(k-1)$-gonal
$3$-independent face, and the proof ends by induction on $k$.
Therefore it is sufficient to show the existence of such an edge.

\begin{figure}[htbp]
\begin{center}
\includegraphics[bb=0 0 280 121,width=4.3cm]{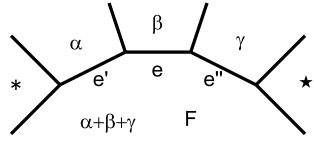}
\end{center}
\caption{A $4$-colored edge $e$ of a $3$-independent face $F$.}
\label{pfProp.bmp}
\end{figure}

In Figure \ref{pfProp.bmp} we assume that $F$ becomes $2$-independent 
after doing $\natural^D$ along the edge $e$.
Then an adjacent face of $F$ which is painted as $\beta$ must be unique,
and the other faces are painted by $\alpha$ and $\gamma$ alternatively
such as $\ast = \gamma$, $\cdots$, $\star = \alpha$.
In particular when $k=4$ (or even), the contradiction arises
because $\ast= \star$.
When $k \ge 5$ and this situation arises, we can do the Dehn surgery
$\natural^D$ along the edge $e'$ (or $e''$) preserving the $3$-independence
of $F$.
\end{proof}

\begin{remark}
In the proof of Proposition \ref{PN1} when we ignore the coloring of $P$,
the Dehn surgery $\natural^D$ can be continued until a triangle appears
for all faces, and then leads to a well-known fact that  
``Each $3$-polytope is bistellarly equivalent to each other'' 
or equivalently ``the $PL$-homeomorphism class of $S^2$ is unique'' 
(cf. \cite{M}). 
\end{remark}

Combining the above proposition and Theorem \ref{TI} and noting the relation 
$\natural^D = \natural \circ (\sharp \Delta^3)$, 
we have the following corollary immediately
(cf. \cite[Theorem 1.10]{N} and \cite[Corollary 4.4]{K}).

\begin{corollary}\label{TN1}
Each $4$-colored polytope $(P^3, \lambda)$ can be constructed from 
$(I^3, \lambda_0)$ and $\Delta^3$ by using three operations 
$\sharp$, $\natural$ and $\natural^{-1}$.
\end{corollary}

Next we consider a construction of all $(\mathbb{Z}_2)^3$-colored polytope.
We recall the basic fact that each $3$-polytope has a face 
which has edges less than six (cf. \cite{G} etc.).
Such a face is called a {\it small face}.
If each small face can be compressed so that the number of faces of $P$
decreases then we can reduce all $(\mathbb{Z}_2)^3$-colored polytopes
to some basic polytopes by induction on the number of faces.
At first we compress $3$-independent small faces.

\begin{proposition}\label{Pro:3-ind}
Let $P$ be a $(\mathbb{Z}_2)^3$-colored polytope except
$\Delta^3$ and $P^3(3)$ as a non-colored polytope.
If there exists a $3$-independent small face of $P$, 
then either $P$ or $\natural^D P$ is quasi-decomposable.
\end{proposition}

\begin{proof}
If there exists a trianglar face of $P$ except $\Delta^3$ and $P^3(3)$ 
then $P$ is decomposable as a non-colored polytope and so 
$(P, \lambda)$ is quasi-decomposable from Proposition \ref{Pro:connD}.
Therefore we can assume that $P$ has no trianglar face.
Let $F$ be a $3$-independent small face of $P$.

(1) When $F$ is a quadrilateral, the situation around $F$
is shown as left of Figure \ref{4-gon1.bmp} where $a_i, b_j \in \mathbb{Z}_2$ 
with $b_2a_3=0$ and at least one of $a_1$ and $b_1$ is nonzero.
By a symmetry we may assume that $a_1=1$.
Since an adjacent triangle does not exist, and
we can always blow down $(\sharp^e)^{-1}P^3(3)$ for $F$
along the horizontal edges (if $a_3b_1=0$) or the
vertical edges (if $b_1=1, b_2=0$), as shown in Figure \ref{cutedge.bmp}.
That is $P= P' \sharp^e P^3(3)$.

(2) When $F$ is a pentagon, the situation around $F$ is shown as
right of Figure \ref{4-gon1.bmp} where $a_i, b_j, c_k \in \mathbb{Z}_2$ with 
$a_2b_3+b_2=1, b_2c_3+b_3=1$ and at least one of $a_1$, $b_1$ and
$c_1$ is nonzero.
We prove that there exists a $0$-sum edge $e$ of $F$ such that 
$F$ is transformed by $\natural^D P$ into a $3$-independent quadrilateral.
Then $\natural^D P$ is quasi-decomposable from the case (1).
Here if the Dehn surgery $\natural^D$ is not allowed then 
$P$ is quasi-decomposable from Lemma \ref{L2}.

\begin{figure}[htbp]
\begin{center}
\includegraphics[bb=0 0 539 230,width=8cm]{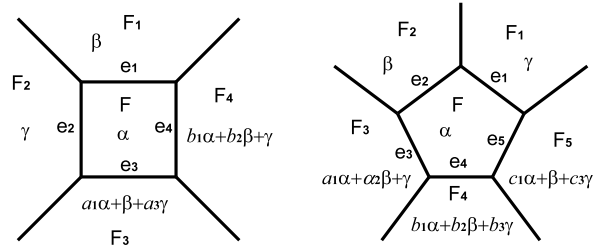}
\end{center}
\caption{$(\mathbb{Z}_2)^3$-colorings around a quadrilateral
and a pentagon.}
\label{4-gon1.bmp}
\end{figure}

i) The case $a_1=1$ (the case $c_1=1$ can be treated similarly).

a) When $a_2=1$, $e_2$ is a $0$-sum edge.
If $c_1=0$ or $b_1+b_2=1$ then the Dehn surgery $\natural^D$ along the 
edge $e_2$
preserves the $3$-independence of $F$ because
the rank of $\{ \lambda(F_1), \lambda(F_3), \lambda(F_4), \lambda(F_5) \}$ 
is three.
If $c_1=1$ and $b_1=b_2=0$ then we have $b_3=1$ and $e_3$ is a $0$-sum edge 
and $\{ \lambda(F_1), \lambda(F_2), \lambda(F_5) \}$ is linearly independent.
If $c_1=1$ and $b_1=b_2=1$ then we have $c_3=1$ and $e_1$ is a $0$-sum edge 
and $\{ \lambda(F_2), \lambda(F_3), \lambda(F_4) \}$ is linearly independent.
In all cases the Dehn surgery $\natural^D$ along a certain $0$-sum edge 
preserves the $3$-independence of $F$.

b) If $a_2=0$ then we have $b_2=1$ and $b_3+c_3=1$.
Therefore we obtain $\sum_{e_4}\lambda(F)= (b_1+c_1) \alpha$ and
$\sum_{e_5} \lambda(F)= (b_1+c_1+1) \alpha$, so
either $e_4$ or $e_5$ is a $0$-sum edge.
Since $\{ \lambda(F_1), \lambda(F_2), \lambda(F_3) \}$ is linearly 
independent, the Dehn surgery $\natural^D$ along $e_4$ or $e_5$
preserves the $3$-independence of $F$.

ii) The case $a_1=c_1=0$ and $b_1=1$.
We have $a_2b_3+b_2=1$, $b_2c_3+b_3=1$ and 
$\{ \lambda(F_1), \lambda(F_2), \lambda(F_4) \}$ is linearly independent.
In this case since $\sum_{e_3} \lambda(F)= 
(a_2+b_2+1) \beta + (b_3+1) \gamma =a_2(1+b_3) \beta +b_2c_3 \gamma$ and
$\sum_{e_5}\lambda(F)=(b_2+1)\beta+(b_3+c_3+1)\gamma
=a_2b_3 \beta +c_3(1+b_2)\gamma$,
either $e_3$ or $e_5$ is a $0$-sum edge
(if $a_2=c_3=1$ then $b_2+b_3=1$).
Then the Dehn surgery $\natural^D$ preserves the $3$-independence of $F$.
\end{proof}

\begin{remark}\label{Rm:3-ind}
In the above proposition, except an irregular quasi-decomposition
because of the prohibition of $\natural^D$,
we have the fact that each $3$-independent small face $F$ is compressible:
such as $P=P' \sharp \Delta^3$ when $F$ is a triangle, 
$P=P' \sharp^e P^3(3)$ or $P' \sharp P^3(3)$ when $F$ is a quadrilateral
and $P=\natural^D (P' \sharp^e P^3(3))$ or $\natural^D (P' \sharp P^3(3))$
when $F$ is a pentagon respectively.
In all cases the number of faces decreases by this decomposition.
\end{remark}

\begin{proposition}\label{Pro:2-ind}
Let $P$ be a $(\mathbb{Z}_2)^3$-colored polytope except 
$\Delta^3$, $P^3(3)$ and $I^3$ as a non-colored polytope.
If there exists a $2$-independent small face of $P$ then
either $P$ or $\natural^{-1}P$ is quasi-decomposable.
\end{proposition}

\begin{proof}
If there exists a triangular face of $P$ except $\Delta^3$ and $P^3(3)$ then 
$P$ is decomposable as a non-colored polytope and so 
$(P, \lambda)$ is quasi-decomposable from Proposition \ref{Pro:connD}.
Therefore we can assume that $P$ has no trianglar face.
Let $F$ be a $2$-independent small face of $P$.
We notice that the inverse surgery $\natural^{-1}$ is allowed
in the category of $(\mathbb{Z}_2)^3$-colored polytopes
when $(P, \lambda)$ is not quasi-decomposable by Lemma \ref{L2}.

\begin{figure}[htbp]
\begin{center}
\includegraphics[bb=0 0 560 181,width=10cm]{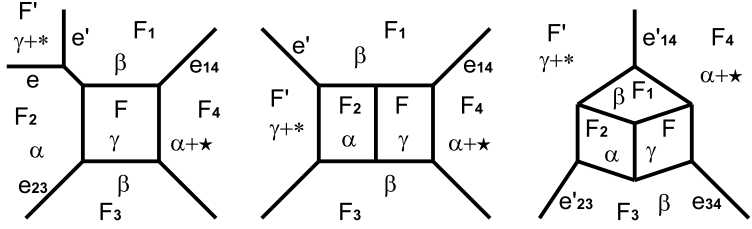}
\end{center}
\caption{The compression of a $2$-independent quadrilateral.}
\label{4-gonC.bmp}
\end{figure}

(1) When $F$ is a quadrilateral,
the number of quadrilaterals adjacent to $F$ is at most two
because $P \neq I^3$ and
the situation around a $2$-independent quadrilateral $F$ is shown as
Figure \ref{4-gonC.bmp} where $\star=\beta$ or $0$.
If $F_1$ and $F_2$ are quadrilateral (the third diagram), 
then $P$ can be decomposed into the connected sum of
a certain polytope $P'$ and $I^3$ with a certain coloring
because the $1$-skeleton of $P$ becomes disconnected after cutting
three edges $e'_{14}, ~e'_{23}$ and $e_{34}$.
If $F_2$ is quadrilateral and $F_1$ and $F_3$ have both at least five edges
(the second diagram) then we can do the surgery $\natural^{-1}$
along edges $e'$ and $e_{14}$, and lead to the third diagram,
so $\natural^{-1}P$ is decomposable (or $P$ is quasi-decomposable),
more precisely $(P, \lambda)=(P', \lambda') \sharp^e (I^3, \lambda'')$ 
for some $(\mathbb{Z}_2)^3$-colored polytope $(P', \lambda)$ and 
a $(\mathbb{Z}_2)^3$-coloring $\lambda''$ on $I^3$.
If $F$ is not adjacent to a quadrilateral (the first diagram)
then we can do the surgery $\natural^{-1}$ along edges $e$ and $e_{23}$
and lead to the second diagram i.e., $\natural^{-1}P$ is quasi-decomposable.

(2) When $F$ is a pentagon,
the situation around a $2$-independent pentagon $F$ is
shown as the first diagram in Figure \ref{5-gonRN.bmp}.
We can assume that $P$ has no triangle and no quadrilateral 
from (1).
We do the surgery $\natural^{-1}$ along the edges $e$ and $e'$
and divide $F$ into a triangle and a quadrilateral (the second diagram).
Therefore $\natural^{-1}P$ is quasi-decomposable from Proposition 
\ref{Pro:connD}.
\end{proof}

\begin{figure}[htbp]
\begin{center}
\includegraphics[bb=0 0 576 211,width=8cm]{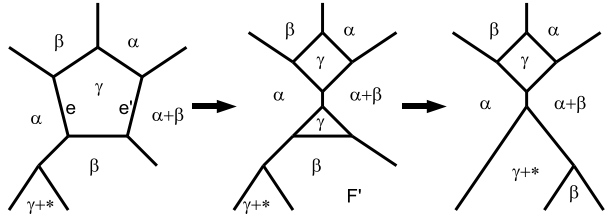}
\end{center}
\caption{The compression of a $2$-independent pentagon.}
\label{5-gonRN.bmp}
\end{figure}

\begin{remark}
When $F$ is a pentagon in the proof of the above proposition,
although the compression of the triangle of $\natural^{-1}P$ does not 
change the number of faces compared with the beginning,
$F$ is tranformed into a quadrilateral by this step 
(see the third diagram).
Then we apply the argument (2) in the proof of Proposition \ref{Pro:2-ind}
to the quadrilateral so that the number of faces in the resulting
polytope is one less than the number of faces in $P$.
\end{remark}

In consequence of Propositions \ref{Pro:3-ind} and \ref{Pro:2-ind},
we can reduce any $(\mathbb{Z}_2)^3$-colored polytope to
$\Delta^3$, $I^3$ and $P^3(3)$ with a certain coloring
by using the surgeries $\natural^{-1}, ~\natural^D =
(\sharp^{-1} \Delta^3) \circ \natural^{-1}$
(without  $\natural$) and the inverses of connected sums 
$\sharp, ~ \sharp^e(= \natural \circ \sharp)$.
From Examples \ref{Ex:prism} and \ref{Ex:cube} the possible colorings on 
$P^3(3)$ (resp. $I^3$) are only three (resp. four) types.
We notice that $(I^3, \lambda_i)= (P^3(3), \lambda_i) \sharp^e 
(P^3(3), \lambda_i) ~(i=1,2,3)$ along vertical edges 
and $(P^3(3), \lambda_3) = \Delta^3 \sharp \Delta^3$.
Therefore there exist four basic $(\mathbb{Z}_2)^3$-colored polytopes:
$(I^3, \lambda_0)$ ($3$-colored), $\Delta^3$ (orientable $4$-colored), 
$(P^3(3), \lambda_1)$ (non-orientable $4$-colored) and 
$(P^3(3), \lambda_2)$ (non-orientable $5$-colored).
Since the surgeries $\natural$ and $\natural^{-1}$ preserves 
the number of colors of faces, and the connected sum $\sharp$ increases 
the number of faces,
it is clear that these four polytopes can not be constructed
from others by using only $\sharp$, $\natural$ and $\natural^{-1}$.
Therefore we have,

\begin{theorem}\label{NN}
Each $(\mathbb{Z}_2)^3$-colored polytope $(P^3, \lambda)$ 
can be constructed from
$\Delta^3$, $(I^3, \lambda_0)$, $(P^3(3), \lambda_1)$ 
and $(P^3(3), \lambda_2)$ by using two operations $\sharp$ and $\natural$.
\end{theorem}

The topological translation of the above theorem is Theorem \ref{ThN2}
shown in the introduction.
We restrict the above theorem to $3$- (resp. $4$-)colored polytopes,
and obtain improvements of Theorem \ref{TI} and Corollary \ref{TN1}
as follows.

\begin{corollary}
(1) Each $3$-colored polytope $(P^3, \lambda)$ can be constructed from
$(I^3, \lambda_0)$ by using two operations $\sharp$ and
$\natural$.\\
(2) Each $4$-colored polytope $(P^3, \lambda)$ can be constructed from
$\Delta^3$ and $(I^3, \lambda_0)$ by using two operations $\sharp$ and
$\natural$.
\end{corollary}

\section{Non-decreasing constructions of small covers}

Since the operations $\natural$ and its inverse $\natural^{-1}$ both correspond
to surgeries on small covers, we followed Izmestiev's standpoint in \cite{I}
and used the surgery $\natural$ in the previous section.
However in \cite{LY} L\"{u} and Yu considered a ``non-decreasing'' construction 
by only operations that number of faces is not decreased,
and therefore the use of $\natural$ is prohibited.
To cancel some obstacles they produced new operations $\sharp^{eve}$,
$\sharp^\Delta$ and $\sharp^\copyright_i$,
and showed the following theorem (cf. \cite[Theorem 1.1]{LY}).

\begin{theorem}[L\"{u} and Yu]\label{ThLY}
Each $(\mathbb{Z}_2)^3$-colored polytope $(P^3, \lambda)$ can be 
constructed from $\Delta^3$ and $(P^3(3), \lambda_2)$ by using seven operations
$\sharp$, $\sharp^e$, $\sharp^{eve}$, $\natural^{-1}$, $\sharp^\Delta$,
$\sharp^\copyright_4$ and $\sharp^\copyright_5$.
\end{theorem}

However there is a gap in the proof of their paper (we shall point it out later).
In this section we also consider a non-decreasing construction 
of small covers in their standpoint.
At first we start with $3$-colored polytopes (i.e. linear models).
In \cite{I} Izmestiev claimed that each $3$-colored polytope can be 
constructed from $3$-colored prisms $P^3(2l)$ by using $\sharp$ 
and $\natural^{-1}$ in the proof of Theorem \ref{TI}.
From the relation $P^3(2l)= I^3 \sharp^e \cdots \sharp^e I^3$,
we can obtain a construction of $3$-colored polytopes 
as follows.

\begin{proposition}\label{I2}
Each $3$-colored polytope $(P^3, \lambda)$ can be constructed from 
$(I^3, \lambda_0)$ by using three operations $\sharp$, $\sharp^e$ and 
$\natural^{-1}$.
\end{proposition}

From the above examination we use the operation $\sharp^e$ and $\natural^{-1}$ 
instead of $\natural$ below.
Then we can also use the Dehn surgery $\natural^D$ and its inverse
because of the relations $\natural^D=\sharp^e \Delta^3$ and $(\natural^D)^{-1}=\natural^D$.
Applying Proposition \ref{PN1} to the above proposition, we have,

\begin{proposition}\label{N2}
Each $4$-colored polytope $(P^3, \lambda)$ can be constructed from $\Delta^3$ and
$(I^3, \lambda_0)$ by using three operations $\sharp$, $\sharp^e$ and $\natural^{-1}$.
\end{proposition}

On the other hand there exist some obstacles for the construction of general
$(\mathbb{Z}_2)^3$-colored polytopes.
At first we must prove that Lemma \ref{L2} also holds for the surgery 
$\natural$.

\begin{lemma}\label{L3}
Let $(P, \lambda)$ be a $(\mathbb{Z}_2)^3$-colored polytope and 
$e$ be an edge of $P$ but not an edge of a quadrilateral.
Suppose that the $3$-connectedness of the $1$-skeleton of $P$
is destroyed after doing surgery $\natural$ along the edge $e$.
Then $(P, \lambda)$ is quasi-decomposable.
\end{lemma}

\begin{proof}
In the Figure \ref{surgery.bmp} we assume that the surgery $\natural$ 
destroys the $3$-connectedness of the $1$-skeleton of $P$.
Then there exits a face $F$ such that $F \cap F_2 \neq \emptyset$ and 
$F \cap F_4 \neq \emptyset$ (see Figure \ref{pfSurg.bmp}).
Since neither $F_1$ nor $F_3$ is a quadrilateral, $P \neq P^3(3)$ and 
we can assume that $e_1$ is not adjacent to $e_2$ (i.e., $R \neq Q$).
When $e_1'$ is adjacent to $e_4$ (i.e., $R'=Q'$),
the $1$-skeleton of $P$ becomes disconnected after cutting the three 
non-adjacent edges $e_1, e_2, e'$.
Therefore $P$ is decomposable as a non-colored polytope,
and so $P$ is quasi-decomposable from Proposition \ref{Pro:connD}.
We assume that $e_1'$ is not adjacent to $e_4$ (i.e., $P' \neq Q'$).
We do the inverse surgery $\natural^{-1}$ along the pair of edges
$\{ e_i', e_4 \}$ where $i=3$ when $\lambda(F)$ is either 
$\alpha$ or $\alpha + \beta$, and $i=1$ when it is not so.
If the inverse surgery $\natural^{-1}$ is not allowed then 
$(P, \lambda)$ is quasi-decomposable from Lemma \ref{L2}.
Then the graph of $\natural^{-1}P$ becomes disconnected 
after cutting the three non-adjacent edges $e_2, e_i$ and the edge
to which $e_i'$ and $e_4$ were glued by $\natural^{-1}$,
and $\{ \lambda(F), \lambda(F_2), \lambda(F_i) \}$
is linearly independent.
Therefore $\natural^{-1}P$ is decomposable as a $(\mathbb{Z}_2)^3$-colored 
polytope such as $\natural^{-1}P=P_1 \sharp P_2$, or equivalently
$P=P_1 \sharp^e P_2$ i.e. $P$ is quasi-decomposable.
\end{proof}

\begin{figure}[htbp]
\begin{center}
\includegraphics[bb=0 0 333 179,width=7cm]{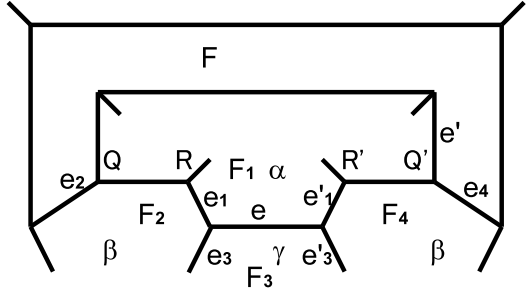}
\end{center}
\caption{The obstacle of the surgery $\natural$.}
\label{pfSurg.bmp}
\end{figure}

\begin{remark}\label{gap}
In \cite{I} Izmestiev used the above lemma only when $F_4$ in Figure
\ref{pfSurg.bmp} is a quadrilateral.
In this case $P$ is always decomposable as a non-colored polytope.
In \cite{LY} L\"{u} and Yu claimed that this argument can be generalized
to every case under the hypothesis of Lemme
\ref{L3} without a proof (cf. \cite[Proposition 2.5]{LY}), 
and proved Theorem \ref{ThLY} using this claim when $F_4$ is a pentagon, too.
However their claim is incorrect (see Figure \ref{counterEX.bmp}).
Although there is a gap in their proof of Theorem \ref{ThLY},
the proof is complemented by using Lemma \ref{L3} 
instead of their key lemma \cite[Proposition 2.5]{LY}.
Furthermore the theorem is improved by replacing $\sharp^\Delta$
with $\sharp^\copyright_3$ as follows: 
Each $(\mathbb{Z}_2)^3$-colored polytope $(P^3, \lambda)$ can be 
constructed from $\Delta^3$ and $(P^3(3), \lambda_2)$ by using seven operations
$\sharp$, $\sharp^e$, $\sharp^{eve}$, $\natural^{-1}$
and $\sharp^\copyright_i$ $(i=3,4,5)$.
\end{remark}

\begin{figure}[htbp]
\begin{center}
\includegraphics[bb=0 0 197 177,width=3cm]{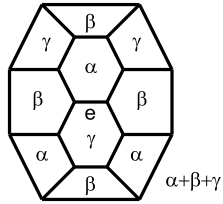}
\end{center}
\caption{A counter example of \cite[Proposition 2.5]{LY}.}
\label{counterEX.bmp}
\end{figure}

From the discussion of the previous section, 
we can reduce each $(\mathbb{Z}_2)^3$-colored polytope $P$
to polytopes which have less faces than $P$
by using the inverses of $\sharp$ and $\sharp^e$ when $P$ has 
a $3$-independent small face, or
a $2$-independent triangle, or
a pair of $2$-independent quadrilaterals adjacent to each other.
Moreover we point out that each $2$-independent pentagon
can be compressed by using the surgery $\natural$ as shown 
in Figure \ref{5-gonR.bmp}.

\begin{figure}[htbp]
\begin{center}
\includegraphics[bb=0 0 576 201,width=10cm]{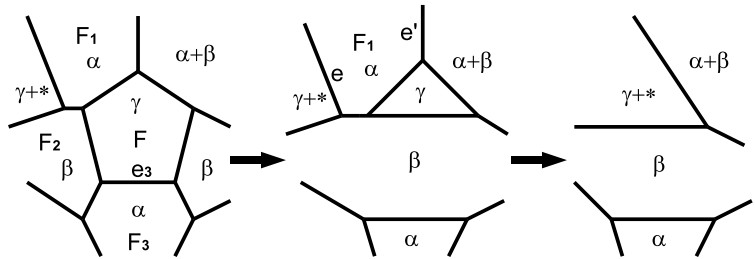}
\end{center}
\caption{Another compression of a $2$-independent pentagon.
In the first diagram we may assume that $F_3$ is not quadrilateral 
by replacing it by $F_2$ if necessary. Then we can do the surgery $\natural$ 
for the edge $e_3$ and transform $F$ into a triangle (second diagram). 
Here when the surgery $\natural$ is not allowed, $P$ is quasi-decomposable
from Lemma \ref{L3}.
Then the triangle can be compressed by $(\sharp^e)^{-1} P^3(3)$ and 
we have $P= \natural^{-1}(P' \sharp^e P^3(3))$ (third diagram).}
\label{5-gonR.bmp}
\end{figure}

In general when colors of two faces on ends of an edge of big faces coincide,
we can do the surgery $\natural$ along this edge and decrease the number
of faces.
Then we can reduce $P$ to $\tilde{P}$ which satisfies the following conditions:
\\
(1) $\tilde{P}$ is not quasi-decomposable,\\
(2) each small face of $\tilde{P}$ is an isolated $2$-independent 
quadrilateral,\\
(3) two colors of faces on ends of every edge which is adjacent to big
faces do not coincide.

\begin{figure}[htbp]
\begin{center}
\includegraphics[bb=0 0 191 217,width=3cm]{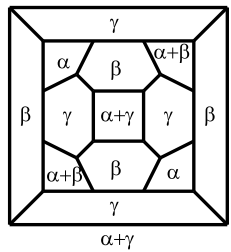}
\end{center}
\caption{Example of an irreducible polytope; truncated octahedron
with a $(\mathbb{Z}_2)^3$-coloring (cf. \cite[Example 2.1]{LY}).}
\label{octahedron.bmp}
\end{figure}

There are many polytopes satisfying the above condition 
(see Figure \ref{octahedron.bmp}).
Obviously such a polytope is irreducible by using the inverses of 
only operations $\sharp$, $\sharp^e$ and $\natural^{-1}$.
Then we need a coloring change operation $\sharp^\copyright_4$ in \cite{LY}.

\begin{definition}[The coloring change $\sharp^\copyright_i$]
The operation in Figure \ref{colorch.bmp} is called the {\it coloring
change} $\sharp^\copyright_i$ for a $2$-independent $i$-gon.
This operation is defined as the connected sum along faces to the
$i$-gonal prism $P^3(i)$ in particular 
$\sharp^\copyright_3= \sharp^\Delta (P^3(3), \lambda_2)$ (see \cite{LY}).
It is clear that $\sharp^\copyright_i$ is invertible because
$(\sharp^\copyright_i)^{-1}=\sharp^\copyright_i$.
\end{definition}

\begin{figure}[htbp]
\begin{center}
\includegraphics[bb=0 0 372 134,width=6cm]{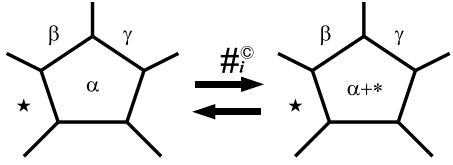}
\end{center}
\caption{The coloring change $\sharp^\copyright_i$ for 
$2$-independent $i$-gon.}
\label{colorch.bmp}
\end{figure}

By using the operation $\sharp^\copyright_4$, 
we can change a color of each $2$-independent 
quadrilateral, and compress it by the surgery $\natural$.
Moreover the $3$-colored cube $(I^3, \lambda_0)$ is obtained by
this operation from other basic polytopes such as 
$\sharp^\copyright_4 (I^3, \lambda_i)$ ($i=1$ or $3$).
Therefore we have an improvement of Theorem \ref{ThLY} as follows.

\begin{theorem}\label{N3}
Each $(\mathbb{Z}_2)^3$-colored polytope $(P^3, \lambda)$
can be constructed from
$\Delta^3$, $(P^3(3), \lambda_1)$ and $(P^3(3), \lambda_2)$ 
by using four operations
$\sharp$, $\sharp^e$, $\natural^{-1}$ and $\sharp^\copyright_4$.
\end{theorem}

The topological translations of Propositions \ref{I2}, \ref{N2} and
Theorem \ref{N3} are stated in Theorem \ref{ThLast}.

\section{Locally standard $2$-torus manifolds over $D^3$}

In this section we shall give a remark for $2$-torus manifolds.
A {\it $2$-torus manifold} $M^n$ is an $n$-dimensional closed smooth manifold
with an effective action of $(\mathbb{Z}_2)^n$ (see \cite{L}, \cite{LM}
for detail).
If the action is locally standard then the orbit space $Q$ is a nice
manifold with corners.
When $Q$ is a simple convex polytope, $M$ is a small cover.

We consider the case that $Q$ is a $3$-dimensional disc $D^3$ with a simple
cell decomposition of the boundary $\partial D^3$, i.e.
{\it a locally standard $2$-torus manifold over $D^3$}.
This class is a little wider than $3$-dimensional small covers.
In fact the $1$-skeleton of $Q$ is a $2$-connected $3$-valent planner graph.
This graph is simple and $3$-connected if and only if $Q$ is a simple convex
polytope.
In this category there is little obstacle of surgeries.
Therefore it becomes easy to discuss in previous sections.

\begin{example}
In Figure \ref{2-gon.bmp} we show the characteristic functions of
$S^3$ with a standard $(\mathbb{Z}_2)^3$-action and
three different $(\mathbb{Z}_2)^3$-colorings of the $2$-sided prism $P^3(2)$,
respectively.
Then the associated $2$-torus manifolds $M(P^3(2), \lambda_i)$ are 
non-equivariantly homeomorphic to
$S^1 \times S^2$, $S^2$-bundle over $S^1$ characterized by the
conjugation $z \mapsto \bar{z}$ on $S^2=\mathbb{C}P^1$ and
$S^1 \times S^2$ according to $i=0,1,2$.
We denote $M(P^3(2), \lambda_1)$ by $S^1 \times_{\mathbb{Z}_2} S^2$ 
where a $\mathbb{Z}_2$-action on $S^1 \times S^2$ is given as 
follows: $t \cdot (s, z)=(-s, \bar{z})$.
\end{example}

\begin{figure}[htbp]
\begin{center}
\includegraphics[bb=0 0 463 132,width=8cm]{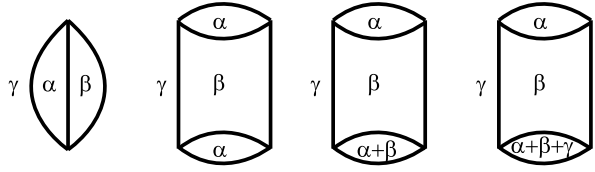}
\end{center}
\caption{The $(\mathbb{Z}_2)^3$-colored simple cell decompositions of $D^3$; 
$\oslash$, $(P^3(2), \lambda_0)$, $(P^3(2), \lambda_1)$ and 
$(P^3(2), \lambda_2)$.}
\label{2-gon.bmp}
\end{figure}

\begin{remark}\label{Rm2}
We can easily verify the following relations:\\
(1) $\sharp \oslash$ is trivial and $\sharp^e \oslash = \natural$.\\
(2) $\sharp P^3(2)$ (or $\sharp^e P^3(2)$ along the horizontal edge) 
is a blow up shown in Figure \ref{bl2.bmp} and
$\sharp^e P^3(2)$ (along the vertical edge) is trivial.\\
(3) $\natural^2 (I^3, \lambda_0)= (P^3(2), \lambda_0)$ and
$\natural (P^3(2), \lambda_0)= \oslash$.\\
(4) $\natural (P^3(3), \lambda_1)= (P^3(2), \lambda_1)$.\\
(5) $\natural^D \Delta^3= (P^3(2), \lambda_2)$.
\end{remark}

\begin{figure}[htbp]
\begin{center}
\includegraphics[bb=0 0 437 96,width=7cm]{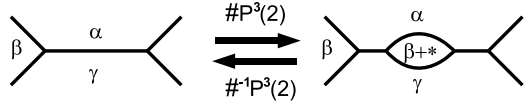}
\end{center}
\caption{The blow up $\sharp P^3(2)$ and its inverse.
In particular $\sharp (P^3(2), \lambda_0)$ (when $*=0$) is the inverse surgery
$\natural^{-1}$ along a pair of adjecent edges.
In \cite{K} the blow down $\sharp^{-1} P^3(2)$ is written by $\natural^0$.}
\label{bl2.bmp}
\end{figure}

We notice that if $Q$ is not $3$-connected then $Q$ is decomposable
as a $(\mathbb{Z}_2)^3$-colored cell decomposition of $D^3$.
Therefore applying the above remark (3), (4) and (5) to Theorem \ref{NN}
we obtain the following corollary immediately.

\begin{corollary}
Each $(\mathbb{Z}_2)^3$-colored cell decomposition of $D^3$ 
can be constructed from $\Delta^3$, 
$(I^3, \lambda_0)$, $(P^3(3), \lambda_1)$ and $(P^3(3), \lambda_2)$ 
by using two operations $\sharp$ and $\natural$.\\
\end{corollary}

In the category of $2$-torus manifolds, there is little obstacle for
surgeries and blow downs.
Therefore we need not consider the case that surgeries are not allowed
(e.g. Lemmas \ref{L2} and \ref{L3}), and obtain the following theorem.

\begin{theorem}
(1) Each $3$-colored cell decomposition of $D^3$ can be
constructed from $\oslash$ by using the inverse surgery $\natural^{-1}$.\\
(2) Each $4$-colored cell decomposition of $D^3$ can be
constructed from $\oslash$ by using the inverse surgery $\natural^{-1}$, 
the Dehn surgery $\natural^D (=\sharp^e \Delta^3)$ and the blow up 
$\sharp \Delta^3$.\\
(3) Each $(\mathbb{Z}_2)^3$-colored cell decomposition of $D^3$ can be
constructed from $\oslash$ by using the inverse surgery $\natural^{-1}$
and connecting $\Delta^3$, $(P^3(2), \lambda_1)$ and
$(P^3(3), \lambda_2)$ by the operations $\sharp$ and $\sharp^e$.
\end{theorem}

\begin{proof}
Let $(Q, \lambda)$ be a $(\mathbb{Z}_2)^3$-colored cell decomposition of $D^3$
but not $\oslash$.
If a $2$-gonal face appears in the following discussion then 
$(P^3(2),\lambda_1)$ is separated from $Q$ or we do the surgery $\natural$ 
and a $2$-gon is compressed immediately.

(1) Each $3$-colored cell decomposition except $\oslash$ can be done 
the surgery $\natural$ and decrease the number of faces.

(2) In the proof of Proposition \ref{PN1} 
the Dehn surgery $\natural^D$ can be continued until a triangle appears
because there is no obstacle of $\natural^D$.
Therefore each $4$-colored cell decomposition of $D^3$ can be
reduced to a $3$-colored cell decomposition by using $\natural^D$ and
the blow down $\sharp^{-1} \Delta^3$.

(3) In the proofs of Propositions \ref{Pro:3-ind} and \ref{Pro:2-ind}
we need not consider the quasi-decomposition by prohibition of surgeries.
When $Q$ has a $3$-independent small face, $Q$ can be reduced
by the blow downs $\sharp^{-1} \Delta^3$, $\sharp^{-1} P^3(3)$, 
$(\sharp^e)^{-1} P^3(3)$ and the Dehn surgery $\natural^D$.
When $Q$ has a $2$-independent triangle, $Q$ can be reduced by the blow downs
$\sharp^{-1} P^3(3)$ and $(\sharp^e)^{-1} P^3(3)$ (along the horizontal edge).
Since each $2$-independent quadrilateral (or pentagon)
has a $3$-colored edge, we can do the surgery $\natural$ along this edge
in this category and decrease the number of faces.
Therefore we can reduce $Q$ to the basic polytopes $\Delta^3$, $P^3(3)$
and $P^3(2)$ by using $\natural$ and inverses of $\sharp$ and $\sharp^e$.
From the relations (3), (4) and (5) in Remark \ref{Rm2},
$(P^3(2), \lambda_0)$, $(P^3(2), \lambda_2)$ and $(P^3(3), \lambda_i)$
($i=1,3$) can be constructed from others.
Here $\sharp$ (or $\sharp^e$) and $\natural^{-1}$ (or $\natural^D$) 
are commutative in this category such as 
$\sharp (P^3(2), \lambda_2)= \natural^D \circ \sharp \Delta^3$, 
$\sharp (P^3(3), \lambda_1)= \natural^{-1} \circ \sharp (P^3(2), \lambda_1)$
and so on.
Then the proof is complete.
\end{proof}

The topological translation of the above theorem is stated in Theorem
\ref{NN2}.

\vspace{10pt}
\textbf{Acknowledgments.}
Finally the author would like to thank Professor M. Masuda for his advice
and stimulating discussions.


\bibliographystyle{amsplain}

\end{document}